\documentclass[
 aip,
 amsmath,amssymb,
 reprint
]{revtex4-1}
\usepackage{graphicx}
\usepackage{dcolumn}
\usepackage[mathcal]{euscript}
\usepackage{bm}
\usepackage[utf8]{inputenc}
\usepackage[T1]{fontenc}
\usepackage{mathptmx}
\usepackage{etoolbox}
\usepackage{amsfonts}
\usepackage{amssymb}
\usepackage{amsthm}
\usepackage{amssymb}
\usepackage{float}
\usepackage[all]{xy}
\usepackage{graphicx}
\usepackage{hyperref}
\usepackage{bbm}
\usepackage{xcolor}
\usepackage[caption=false]{subfig}

\newcommand{\E}{\mathbb E}

\newcommand{\F}{\mathcal F}
\newcommand{\T}{\mathcal T}

\newcommand{\PP}{\mathcal P}

\newcommand{\K}{\mathcal K}
\newcommand{\A}{\mathcal A}
\newcommand{\B}{\mathcal B}
\newcommand{\x}{\bar x}
\newcommand{\y}{\bar y}

\theoremstyle{plain}
\newtheorem{theorem}{Theorem}[section]

\newtheorem{lemma}[theorem]{Lemma}

\makeatletter
\def\@email#1#2{%
 \endgroup
 \patchcmd{\titleblock@produce}
  {\frontmatter@RRAPformat}
  {\frontmatter@RRAPformat{\produce@RRAP{*#1\href{mailto:#2}{#2}}}\frontmatter@RRAPformat}
  {}{}
}%
\makeatother
\begin{document}

\preprint{AIP/123-QED}

\title[Arbitrarily accurate coarse graining]{Arbitrarily accurate, nonparametric coarse graining with  Markov renewal processes and the Mori-Zwanzig formulation}
% Force line breaks with \\
\author{David Aristoff}
\affiliation{Colorado State University, Fort Collins, CO, 80523, USA}
\thanks{Author to whom correspondence should be addressed: \url{aristoff@colostate.edu}}
\author{Mats Johnson}
\affiliation{Colorado State University, Fort Collins, CO, 80523, USA}
\author{Danny Perez}
 \affiliation{Theoretical Division T-1, Los Alamos National Laboratory, Los Alamos, NM, 87545, USA}

\date{\today}% It is always \today, today,

\begin{abstract}
Stochastic dynamics, such as molecular dynamics, are important in many scientific applications. However, summarizing and analyzing the results of such simulations is often challenging, due to the high dimension in which simulations are carried out, and consequently to the very large amount of data that is typically generated.
Coarse graining is a popular technique for addressing this problem by providing compact and expressive representations. Coarse graining, however, potentially comes at the cost of accuracy, as dynamical information is in general lost when projecting the problem in a lower dimensional space. This article shows how to eliminate coarse-graining error using two key ideas. First, we represent coarse-grained dynamics as a Markov renewal process. Second, we outline a data-driven, non-parametric Mori-Zwanzig approach for computing jump times of the renewal process. Numerical tests on a small protein illustrate the method.
\end{abstract}

\maketitle

% \begin{quotation}
% The ``lead paragraph'' is encapsulated with the \LaTeX\ 
% \verb+quotation+ environment and is formatted as a single paragraph before the first section heading. 
% (The \verb+quotation+ environment reverts to its usual meaning after the first sectioning command.) 
% Note that numbered references are allowed in the lead paragraph.
% %
% The lead paragraph will only be found in an article being prepared for the journal \textit{Chaos}.
% \end{quotation}

\section{Introduction}

Stochastic dynamics play a critical role in the study of complex systems across various scientific domains. Molecular dynamics (MD), for instance, simulate the motion of collections of atoms over time. MD simulations have found applications in materials science, chemistry, biology, and physics~\cite{karplus1990molecular,karplus2002molecular,hansson2002molecular,durrant2011molecular,hospital2015molecular,hollingsworth2018molecular}.Analyzing the immense volume of data generated, and navigating the high-dimensional space in which these simulations operate, creates significant challenges. Indeed, a single snapshot of an MD trajectory resides in a continuous 3$N_\mathrm{atom}$-dimensional space, with $N_\mathrm{atom}$ ranging from hundreds to billions. 
This makes model reduction highly desirable. Effective model reduction not only enhances interpretability, but also allows for upscaling results to inform higher-fidelity models.

Yet another challenge comes from metastability~\cite{chong2017path} where stochastic trajectories are confined to small regions of space for long times, punctuated by rare but fast transitions between regions. Metastability is typical in MD, where such regions might represent folded and unfolded states of a protein. In addition to MD, metastable stochastic dynamics arise in climate models~\cite{weare2009particle,webber2019practical,finkel2023data,finkel2023revealing},   granular flows~\cite{seiden2011complexity}, neural evolution~\cite{fingelkurts2004making,haldeman2005critical,hellyer2015cognitive,cordova2017disrupted,naik2017metastability,cavanna2018dynamic}, hydrodynamics~\cite{pomeau1986front}, and power networks~\cite{matthews2018simulating}, in addition to various ordinary differential equations models~\cite{duncan2002metastability,sun1999metastability,estep1994analysis,groisman2018metastability}.

Coarse-graining is a common approach for handling high dimensionality or metastability. It is based on dividing the
original high-dimensional space of {\em microstates} into a discrete set of {\em macrostates}.  Usually, 
the dynamics on the macrostates is modeled as a Continuous Time Markov Chain (CTMC)~\cite{norris1998markov} or a discrete time Markov chain (DTMC)~\cite{norris1998markov,durrett1999essentials}.
% This assumption of Markovian behavior implies that the escape time from a macrostate follows an exponential random variable, and that the next visited macrostate is independent of the escape time. 
In MD, such CTMC models are called chemical reaction networks or Kinetic Monte Carlo models~\cite{angeli2009tutorial,voter2007introduction}, 
and the DTMCs are called Markov State Models~\cite{chodera2014markov}. These Markovian models offer advantages like formal simplicity, compact representation, and ease of use with ready-made algorithms like BKL~\cite{bortz1975new} or Gillespie~\cite{gillespie1977exact} 
for simulation. 

The Markov assumption underlying CTMC and DTMC models is significantly flawed if 
macrostates are not carefully chosen~\cite{husic2018markov}, if temperatures are not sufficiently low~\cite{di2016jump}, or 
if time scales are not long enough. Even with careful choices of all these parameters, some degree of
departure from exact Markovian behavior remain in general~\cite{ross1995stochastic,bremaud2001markov,lelievre2015accelerated,lelievre2020mathematical}. 

Meanwhile, recent findings show the Markov assumption can be weakened, with an arbitrarily accurate representation achievable using Markov Renewal Processes~\cite{cinlar1975exceptional} (MRPs) by simply adjusting a scalar parameter \cite{agarwal2020arbitrarily}. This scalar parameter, called $\tau$ below, is a {\em decorrelation time} chosen to allow the underlying dynamics to periodically reach local equilibrium in the macrostates, inheriting the Markov property at each such time. MRPs differ from Markov processes in only having the Markov property at certain times (called {\em jump times}).
 Despite having formal simplicity, a complete MRP parametrization for $N$ macrostates would require $N^2$ scalars and $N^2$ functions of time. Accurately representing these functions from limited, short-time length data poses a challenge~\cite{agarwal2020arbitrarily}. This article proposes a new technique, rooted in first principles, to efficiently model this MRP with a few $N \times N$ matrices.

\subsection{Contributions}

Below, we propose a compact, data-driven parametrization for the MRP 
model described in~\cite{agarwal2020arbitrarily}. 
Our methods,  
rooted in Mori-Zwanzig theory, are simple and data-driven, and our contributions are practical 
and theoretical. 

On the practical side, we propose a compact, mathematically principled representation 
of the MRP derived from Mori-Zwanzig 
theory. In our formulation, 
the MRP is represented by, and can be generated from, a 
(small) number of memory kernels. These memory kernels are $N \times N$ matrices, where $N$ is the number of macrostates.
We propose a new method to obtain the kernels by
solving a certain linear system comprised of  correlation matrices. Efficient, scalable solvers designed for positive semidefinite systems can then be used to obtain the kernels. (E.g., RPCholesky~\cite{chen2022randomly,diaz2023robust} uses randomized low-rank approximation.) Numerical results 
on alanine dipeptide, a small 
protein, illustrate the promise 
of the method.

On the theoretical 
side, we show that these methods become exact as the number of memory kernels and the decorrelation time grow. This demonstration takes the following steps. To start, we give the first proof that coarse-grained dynamics described in~\cite{agarwal2020arbitrarily} in fact converges to a MRP (Theorem~\ref{thm:RE}). Then, we represent the transition probabilities of the MRP in terms of memory kernels using the discrete Mori Zwanzig equation~\eqref{eq:MZ}. And finally, we prove that equation~\eqref{eq:MZ} is exact (Theorem~\ref{thm:MZeqn}). This equation first appeared in~\cite{cao2020advantages} in a different setting (without the decorrelation). There, it was derived as an approximation of a continuous time Mori-Zwanzig equation. We give the first full derivation of~\eqref{eq:MZ} that shows it is exact for any choice of dynamical lag (we use lag $\tau$ in our setup). As $\tau$ can be significantly longer than the time step of the underlying dynamical integrator, exactness at the discrete time level is important. 

In addition, we provide exact expressions for the memory kernels in terms of an orthogonal dynamics (Appendix~\ref{sec:appendix_MZ}). While these expressions cannot directly be put to practical use, they help lend explainability to the kernels, and could potentially be used to quantify their decay in time. Our novel data-driven method for actually computing the memory kernels, based on the linear solve~\eqref{eq:linear},  can also be explained in terms of inter-macrostate correlations. 

Finally, we
show that our Mori-Zwanzig equation is 
optimal, in the sense that the representation is compact when the MRP representation is almost fully Markovian. We actually 
prove an ideal case of this, showing that all but
one of the memory kernels vanishes 
in the case where the MRP representation 
is in fact Markovian.

This article is organized as follows. We summarize our notation in Table~\ref{tab:symbol-definitions}. In Section~\ref{sec:Markov}, we review
how we discretize the underlying dynamics, following\cite{agarwal2020arbitrarily}. 
In Section~\ref{sec:mz}, we 
introduce the Mori-Zwanzig equation 
and explain how we use it to estimate 
memory kernels nonparametrically from 
short time simulations. We 
also show how the memory kernels 
can be used to infer longer time 
information. In Section~\ref{sec:QSD}, 
we give an outline of our proof that the discretized dynamics 
converges to a MRP (the proof is in Appendix~\ref{sec:appendix_MR}). In Section~\ref{sec:numerics}, we 
illustrate our method on 
alanine dipeptide. We show that 
we can reduce errors arising from 
ordinary spatial discretization, recovering 
accurate dynamics with a relatively 
small number of memory kernels. 
All proofs, including the derivation of the Mori-Zwanzig equation and the proof of convergence to a MRP, are in the Appendix.

\begin{table}[ht]
    \caption{Definitions of symbols used in this work.}
    \label{tab:symbol-definitions}
    \centering
    \begin{tabular}{c | @{\hspace{1em}}l@{}}
        Symbol & Definition \\ [0.5ex] \hline
        $X(t)$ & underlying Markov chain on microstates  \\
        $x$, $y$, $z$ & microstates \\
           $I$, $J$, $L$ & macrostates \\
    $N$ & number of macrostates \\
        $\tau$ & macroscopic time step\\
                $R(t)$ & macroscopic jump process \\
                        $r$, $s$, $t$ & times (multiples of $\tau$, when associated with $R(t)$) \\
            $s_-$, $t_-$ &  preceding times: $s_- = s-\tau$, $t_- = t-\tau$ \\
         $\tau_I$ & decorrelation time in macrostate $I$ 
         \\
        $\eta_I$ & QSD in macrostate $I$ \\
         $\T(s,t)$ & transition probability matrix \\
        $\T(t)$ & transition matrix of renewal process \\
        $\PP(t)$  & jump probability matrix \\
        $\K(t)$  & memory kernel matrix \\
                $C(t)$ & consecutive time in current macrostate \\
                $P$, $Q$ & projector and complementary projector \\
                 $\chi_I$ & characteristic function of macrostate $I$ \\
                 $n,m$ & nonnegative integers
    \end{tabular}
\end{table}

\section{Markov chains and Markov renewal process}\label{sec:Markov}

Throughout, $X(t)$ is an underlying Markov process evolving in a space 
of {\em microstates}. This process 
can be discrete or continuous 
in both time and space. We consider 
a division 
of microstates into finitely many
{\em macrostates} $I$, $J$, etc. 

Our work focuses on a discrete time jump process
$R(t)$ on these macrostates, with time step $\tau$, defined from the underlying process and a set of {\em decorrelation times}, written 
$\tau_I$, $\tau_J$, etc. 
The jumps occur when $X(t)$ spends consecutive time $\tau_J$ 
in some macrostate $J$. 
Specifically, $R(t)$ jumps from $I$ to $J$ at time $t$ if 
$X(t-c)$ 
is in macrostate $J$ for $0 \le c \le \tau_J$. Jumps only occur among 
distinct states ($J \ne I$) and at 
multiples of the time step ($t = n\tau$ for integer $n$). See
Figure~\ref{fig1} for an illustration.

To describe the evolution of $R(t)$, we define $\T_{IJ}(s,t)$ as the probability for $R(t)$ to be in $J$ at time $s+t$, assuming there was a jump into $I$ at time $s$. That is,
\begin{equation}\label{eq:defTst}
\T_{IJ}(s,t) = {\mathbb P}(R(s+t) = J|R(s_-) \ne I,\,R(s) = I),
\end{equation}
where we use the shorthand $s_- = s-\tau$.

The introduction of decorrelation times allows the underlying Markov process 
to reach a local equilibrium 
within each macrostate. 
Conceptually, when $\tau_J$ is large enough, $X(t)$ loses memory of how it entered $J$ by the time 
that $R(t)$ jumps into macrostate $J$. 
This makes $R(t)$ into a MRP, which means 
it has the Markov property at jump 
times~\cite{agarwal2020arbitrarily}. 
Note that $R(t)$ does not retain  
information about what occurs on timescales shorter than the decorrelation times and $\tau$. This is a modeling assumption that may lead to the loss of relevant dynamical information if important transition events occur on such timescales. On the other hand, information loss will be minimal when the typical residence time in a macrostate is much longer than both $\tau$ and the decorrelation time.

\begin{figure}
\centering
    
\includegraphics[width=0.45\textwidth]{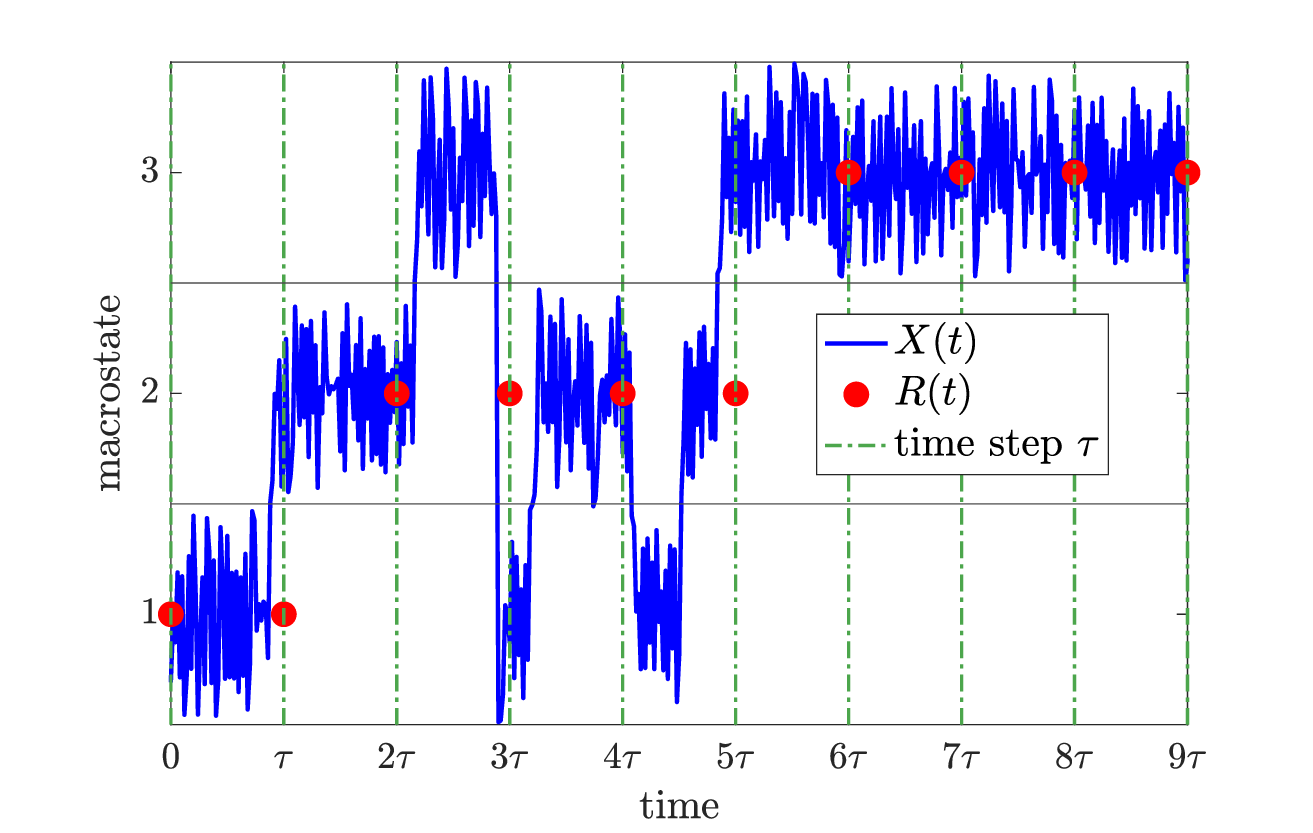}
    \caption{Illustration of $X(t)$ and $R(t)$, 
    with three macrostates labeled $1,2,3$, when the decorrelation times are $\tau_1 = \tau_2 = \tau_3 = \tau$. Solid horizontal 
    lines divide the macrostates. For illustrative purposes, we show an example where $X(t)$ makes several transitions that are not recorded by $R(t)$, due to a failure to decorrelate in macrostates.}
    \label{fig1}
\end{figure}

Assuming that $R(t)$ is in fact a MRP, we can write $\T(s,t) = \T(t)$,
where $\T(t)$ is a standard transition matrix for each $t$. These transition matrices  together satisfy a {\em renewal equation} defined by a {\em jump probability matrix $\PP(t)$}, where $\PP_{IJ}(t)$ is the probability for $R(t)$ to jump from $I$ to $J$ in time $t$:
\begin{equation*}
    \PP_{IJ}(t) = {\mathbb P}(R(s+t) = J|R(s_-) \ne I,\,R(s') = I,\,s\le s' < s + t).
\end{equation*}
The renewal equation is~\cite{cinlar1975exceptional}
\begin{equation}\label{eq:renewal}
\T(t) = \sum_{0< s \le t}  \PP(s)\T(t-s) + \F(t),
\end{equation}
where $\F_{IJ}(t)= \delta_{I = J}\sum_L \sum_{s>t} \PP_{IL}(s)$. Here, $\delta_{I=J} = 1$ if $I = J$, and $\delta_{I=J} = 0$ otherwise.
The time arguments here are multiples of $\tau$, and we continue with this convention for other equations associated with $R(t)$ below.

The Markov renewal framework of~\eqref{eq:renewal} is exact in the limit of 
large decorrelation times (Theorem~\ref{thm:RE}).
Below, 
we outline how to estimate $\T(t)$ 
in a principled, parameter-free way
using Mori-Zwanzig theory. Once $\T(t)$ 
is estimated, equation~\eqref{eq:renewal} can be 
used to compute the jump time distribution 
$\PP(t)$. This provides a principled 
way to describe -- and simulate -- the 
process $R(t)$, which exactly reflects 
the macroscopic behavior of $X(t)$.

Our setup above allows for situations where the decorrelation times 
are state-dependent: there  
is a (potentially different) decorrelation time $\tau_I$ for each macrostate $I$. For simplicity, 
in the numerical 
examples and ensuing discussion in Section~\ref{sec:numerics}, we
take all the 
decorrelation times to be the same and equal to $\tau$, i.e., $\tau_I = \tau$ for each $I$.

\section{Nonparametric estimation of transition probabilities}\label{sec:mz}

Using Mori-Zwanzig theory, 
\begin{equation}\label{eq:MZ}
    \T(t) = \sum_{0<s\le t} \K(s)\T(t-s),
\end{equation}
where $\K(s)$ are {\em memory 
kernels} that can be estimated from data, as we describe below. 
Equation~\eqref{eq:MZ} was derived as an approximation of a continuous-time Mori Zwanzig equation in~\cite{cao2020advantages}, 
while different discrete time Mori-Zwanzig equations have been described in~\cite{darve2009computing,lin2021data}.
We will give a short proof of exactness of~\eqref{eq:MZ} in  Appendix~\ref{sec:appendix_MZ} (Theorem~\ref{thm:MZeqn}), and 
provide more 
details on the 
memory kernel structure there.

Equations~\eqref{eq:renewal} and~\eqref{eq:MZ} appear superficially similar but are quite different. While $\PP(s)$ defines jump probabilities of the MRP, $\K(s)$ involves quantities associated to a so-called {\em orthogonal dynamics}. Roughly speaking, this dynamics describes  situations where $X(t)$ transitions between macrostates without decorrelating in them. 
We arrived at~\eqref{eq:MZ} by choosing a Mori-Zwanzig projector that leads to very compact representations (i.e., fast time decay of memory kernels) when $R(t)$ is nearly Markovian. Indeed, in Appendix~\ref{sec:appendix_MZ}, we show that if $R(t)$ is 
actually Markovian, only one memory kernel is nonzero, $\K(s) = 0$ for $s > \tau$. Meanwhile, if $R(t)$ is Markovian, then 
$\PP(s)$ is geometric in $s$ with rates in inverse proportion to the mean jump times between macrostates
(resulting in slow decay of $\PP(s)$ for large mean jump times).

While equation~\eqref{eq:MZ} could be used to solve for the memory kernels directly given enough sampling~\cite{cao2020advantages}, we find that the following setup is 
superior in practice. In order to nonparametrically estimate 
$\K(t)$, we introduce a loss function
\begin{equation}\label{eq:loss}
{\mathcal L}(\K) = \sum_{t\le t_{max}}\left\|\T(t) -\sum_{0<s\le \min\{t,t_{mem}\}} \K(s)\T(t-s)\right\|^2,
\end{equation}
where $t_{mem}$ is a cutoff time for the memory matrices, $t_{max}$ is a 
cutoff time for the transition matrices, and $\| \cdot\|$ represents the Frobenius norm. 

By setting the gradient of the loss function equal to zero, we get the 
following symmetric positive semidefinite linear system that can 
be solved for the memory matrices (see Appendix~\ref{sec:appendix_linear}):
\begin{equation}\label{eq:linear}
\sum_{0< s \le t_{mem}} \K(s)\A(s,t) = \B(t), \quad 0 < t \le t_{mem},
\end{equation}
where $\A$ and $\B$ are the correlation matrices
\begin{align}\begin{split}\label{eq:corr_matrices}
 \A(s,t) &= \sum_{r \le t_{max}} \T(r-s)\T(r-t)^T, \\
 \B(s) &= \sum_{r\le t_{max}} \T(r)\T(r-s)^T,
 \end{split}
\end{align}
and where by convention $\T(s) = 0$ for $s<0$. (Various regularizations, including ridge regression that penalizes the Frobenius norms of the memory kernels, 
can easily be applied if desired.) 

%replacing $\A(s,t)$ in~\eqref{eq:linear} with $\A(s,t) + \lambda I\delta_{s=t}$, where 
%$I$ is the identity matrix and 
%$\lambda > 0$ is a regularization 
%parameter. 
%This has the effect of 

The memory kernels $\K(t)$ can then 
be obtained as follows. First, we 
can estimate $\T(t)$ for $t \le t_{max}$ from data of the underlying 
Markovian dynamics. Then, we can 
estimate the matrices $\A$ and $\B$ in~\eqref{eq:corr_matrices}. 
Finally, we solve the linear system~\eqref{eq:linear} to 
obtain $\K(t)$ for $0<t \le t_{mem}$. 

With the memory kernels in hand, the 
transition probabilities can be 
estimated by repeatedly applying the equation
\begin{equation}\label{eq:inferT}
\T(t) \approx \sum_{0<s\le \min\{t,t_{mem}\}} \K(s)\T(t-s),
\end{equation}
while incrementally increasing $t$.
Note that this allows for estimation up to any time, including beyond $t_{max}$.
The memory kernels carry $N^2 k$ 
entries in total, with $N$ 
the number of macrostates and $k$ 
the number of memory kernels. We 
find good results even with a 
relatively small number of kernels; see Section~\ref{sec:numerics}. Once $\T(t)$ is 
in hand, $\PP(t)$ can be computed by unrolling 
the renewal equation~\eqref{eq:renewal}.

In Appendix~\ref{sec:Markov}, we show that if $R(t)$ is actually 
a Markov chain -- that is, if it has 
the Markov property at {\em all} times, not just at jump times -- 
then $\K(t)= 0$ for $t > \tau$. 
In this case, 
$\T(n \tau) = \K(\tau)^n = \T(\tau)^{n}$, 
and the estimation of the system only depends on the underlying Markov chain dynamics at lag $\tau$. 
Equation~\eqref{eq:inferT} provides an 
extension of 
this to allow for 
non-Markovian behavior.

Other methods for 
estimating memory kernels 
have been recently described in~\cite{cao2020advantages,lin2022regression,dominic2023building,dominic2023memory}. We find that our method 
significantly outperforms applying a direct solve~\cite{cao2020advantages} in equation~\eqref{eq:MZ}, 
while inheriting the simplicity of least squares~\cite{lin2022regression}, 
and interpretability 
in terms of time correlation 
matrices.

\section{Quasistationary distributions, and 
convergence to a Markov renewal process}\label{sec:QSD}

For large enough decorrelation times, 
the underlying process reaches a local equilibrium each time that $R(t)$ 
makes a jump, leading 
to a Markov property for 
$R(t)$. We now make this 
precise using {\em quasistationary distributions} (QSDs).

\begin{figure*}
\centering
\includegraphics[trim=4.25cm 0cm 4.25cm 0cm, width=0.95\textwidth, clip]{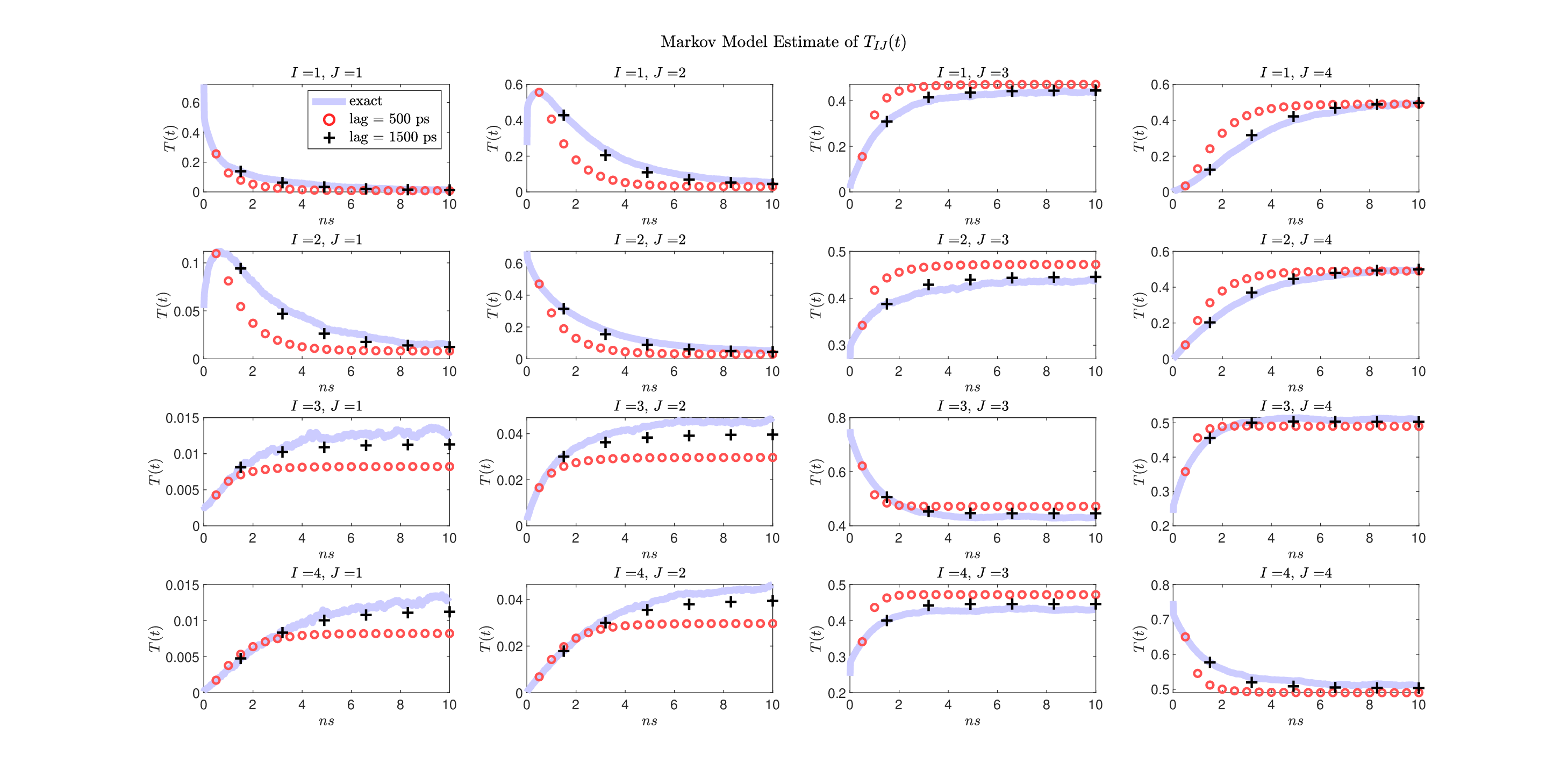}
    \caption{Building a Markov model for alanine dipeptide, using states defined through PCCA. 
    Except at very long lags, the Markov model is considerably less accurate than what we obtain with our methods (see Figure~\ref{fig:mz_comp}). This is because a simple coarse-graining of $X(t)$ into these states is not sufficiently Markovian. Each Markov model is based on a single transition matrix, computed from counts of transitions of $X(t)$ between macrostates at the specified lag time. Transitions at longer lags are computed using powers of this single matrix.} 
    %\caption{Attempting to build a Markov model for alanine dipeptide, using the ``bad'' states of Figure~\ref{fig:states}. Except at very long lags, the Markov model performs poorly. This is because a simple coarse-graining of $X(t)$ into these states is highly non-Markovian. Each Markov model is based on a single transition matrix, computed from counts of transitions of $X(t)$ between macrostates at the specified lag time. Transitions at longer lags are computed using powers of this single matrix.} 
    \label{fig:msm_comp}
\end{figure*}

The QSD of $X(t)$
in $I$ is defined 
by the condition that if $X(t)$ is 
initially distributed as the QSD in $I$, 
then conditionally on staying in $I$, it remains distributed as the QSD. 
Writing $\eta_I$ for the QSD in $I$, 
\begin{equation}\label{eq:QSD}
\eta_I(\cdot) = \int \eta_I(dx){\mathbb P}(X(t) \in \cdot | X(0) = x,\, X(s) \in I,\, s \le t),
\end{equation}
where the variable $x$ represents microstates of $X(t)$.

Under mild assumptions~\cite{collet2013quasi,champagnat2023general},
\begin{equation}\label{eq:QSD_conv}
\|\eta_I - {\mathbb P}(X(t) \in \cdot \,|\, X(s) \in I,\,s \le t)\| \le c_I \delta_I^t,
\end{equation}
where $c_I$ and $\delta_I < 1$ are constants, 
and the norm is the total variation 
of measures. Informally, given that $X(t)$ remains in macrostate $I$, it converges to $\eta_I$ at a geometric rate. 

In Theorem~\ref{thm:RE} of Appendix~\ref{sec:appendix_MR}, we show that 
\begin{equation}\label{eq:renewal_approx1}
\T(s,t) = O(t\delta^\sigma) + \sum_{0<r \le t}\PP(r)\T(s,t-r) + \F(t),
\end{equation}
where $\PP(t)$ is the jump probability matrix 
of a Markov renewal process, 
$\F_{IJ}(t) = \delta_{I=J} \sum_L \sum_{s>t} \PP_{IL}(s)$, and $\delta = \max_I \delta_I$, $\sigma = \min_I \tau_I$. 
It follows that the transition 
matrices $\T(s,t)$ converge  
to the transition matrices of a Markov 
renewal process defined by the jump 
time distribution $\PP(t)$, at 
a geometric rate in terms of the decorrelation 
times. 

\begin{figure*}
\centering
\subfloat[PCCA states]{
\includegraphics[trim=4.25cm 0cm 4.25cm 0cm, width=0.95\textwidth, clip]{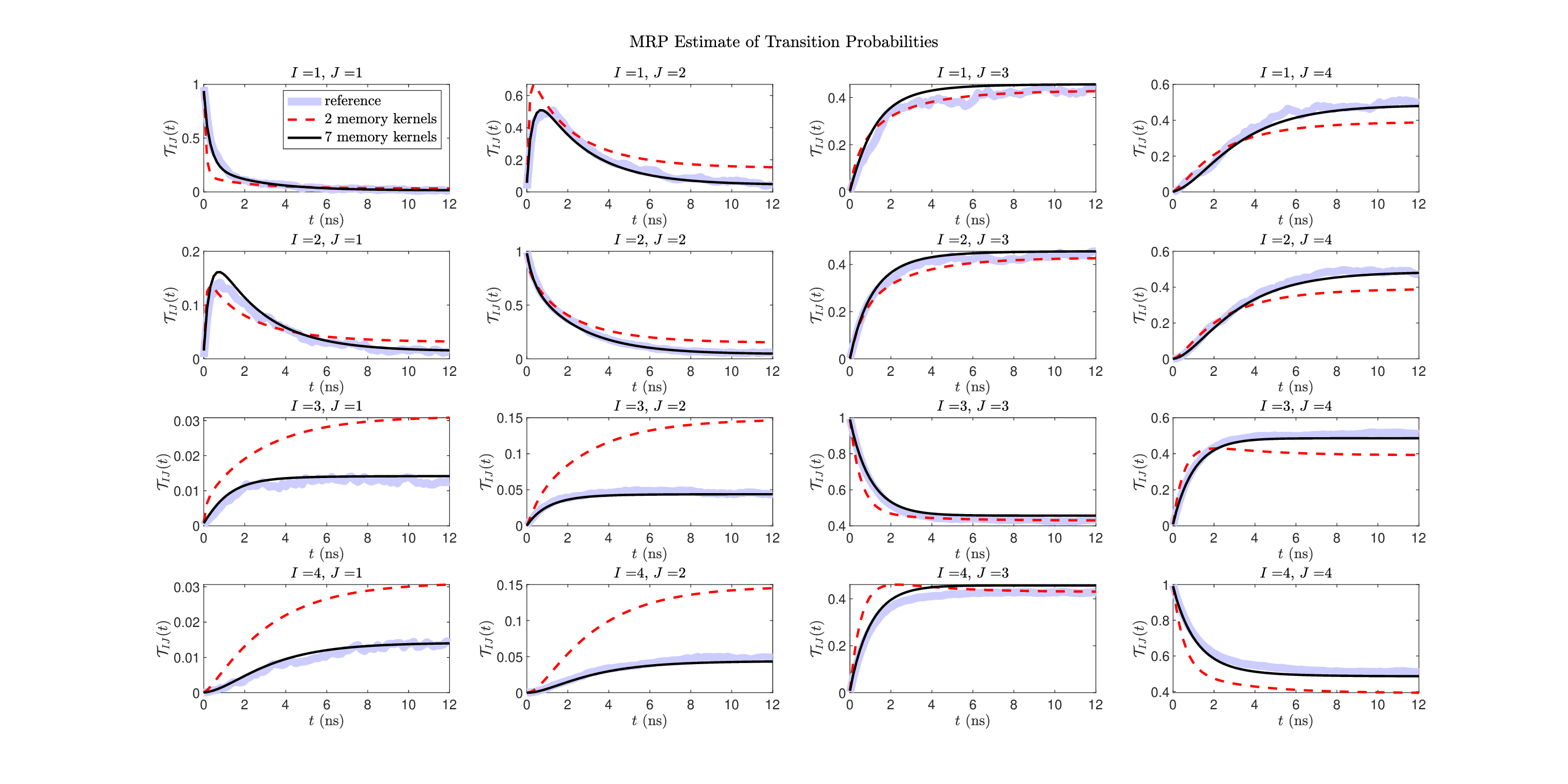}
}

\subfloat[Rectangular states]{
\includegraphics[trim=4.25cm 0cm 4.25cm 0cm, width=0.95\textwidth, clip]{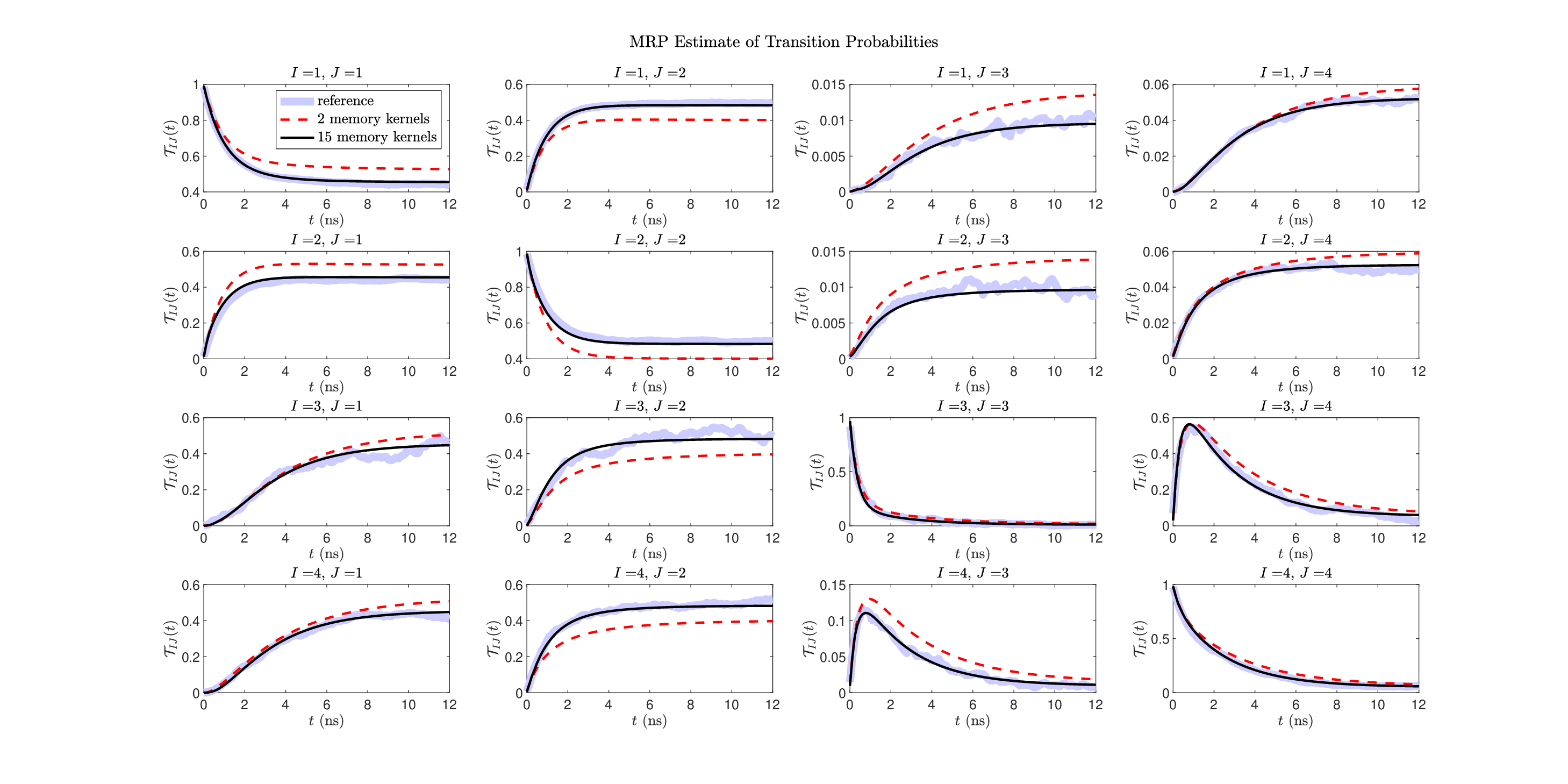}
}
\caption{Results from our method for parametrizing the MRP using (a) states defined by PCCA and (b) equal rectangular states.  A very good parameterization is achieved in each case with $7$ and $15$ memory kernels in (a) and (b) respectively.
Shown in (a) and (b) are transition probabilities inferred using~\eqref{eq:inferT} with a smaller number of memory kernels (dashed line) and a larger number of memory kernels (solid line). The kernels are computed using~\eqref{eq:linear}-~\eqref{eq:corr_matrices} with cutoff time twice the memory length ($t_{max} = 2t_{mem}$). We use the macroscopic time steps, $\tau = 8$ ps (a) and $\tau = 30$ ps (b), that define the ``good'' decorrelation times. (See Figure~\ref{fig:mrp_conv} for the choice of $\tau$ in (a).) Results are clearly improved with the larger number of memory kernels.}
    \label{fig:mz_comp}
\end{figure*}

\section{Numerical results}\label{sec:numerics}

To demonstrate the potential of our method, we apply it
to alanine dipeptide, using an MD trajectory~\cite{agarwal2020arbitrarily} of length about $70$ ms. 
Positions in $\phi$-$\psi$ space were 
saved at every $2$ ps.  The macrostates are either chosen by using PCCA or by dividing $\phi$-$\psi$ space 
into four equal rectangles. While the PCCA states are 
highly metastable, the rectangular states are not. 
A finite spatial discretization limits the accuracy of 
Markov models, as seen in Figure \ref{fig:msm_comp}, 
which shows
that a Markov model does not 
accurately represent the 
discretized alanine dipeptide dynamics, except at 
long timescales.

We use a decorrelation time $\tau_I = \tau$ ps that is the same for all states $I = 1,2,3,4$. These decorrelation times were chosen to be large enough to obtain good numerical accuracy of the renewal equation~\eqref{eq:renewal}; see Figure~\ref{fig:mrp_conv}. 
Then we construct a trajectory $R(t)$ as described in Section~\ref{sec:Markov} (see also Figure~\ref{fig1}), and 
apply our method. The alanine MD trajectory was split in half into a training set and a test (or reference) set. 
We use the former to create our model of $R(t)$, and the latter to create reference results.

To assess our method, we compare 
it with a reference that uses the 
indicated value of $\tau$. The reference results 
are based on simple counts of transitions.
Figure~\ref{fig:mz_comp} compares reference counts with 
our method's estimates of $\T(t)$.
Figure~\ref{fig:mz_conv100} shows the 
error in $\PP(t)$. To mitigate noise effects 
from finite sampling, we use the error measurement
\begin{equation}\begin{split}\label{eq:Cramer}
    \textup{Error} &= \sum_{I,J} \int_0^\infty \left(\int_0^t \frac{[\PP_{IJ}(s)-\hat{\PP}_{IJ}(s)]}{Z_{IJ}}ds\right)^2\frac{\PP_{IJ}(t)}{Z_{IJ}}dt,
    \end{split}
\end{equation}
where $Z_{IJ} = \int_0^\infty \PP_{IJ}(t)\,dt$, and where $\hat{\PP}(t)$ is our estimate on training data, with $\PP(t)$ the reference. This is a slight variation on the Cramer-von Mises criterion~\cite{anderson1962distribution}. 
Figures~\ref{fig:mz_comp} and~\ref{fig:mz_conv100} show that the approach outlined in Section~\ref{sec:mz} gives good agreement with the reference, with just a few memory kernels.

\begin{figure*}
\centering
    \subfloat[$\tau = 2$ps]{
    \includegraphics[trim=4.25cm 0cm 4.25cm 0cm, width=0.95\textwidth, clip]{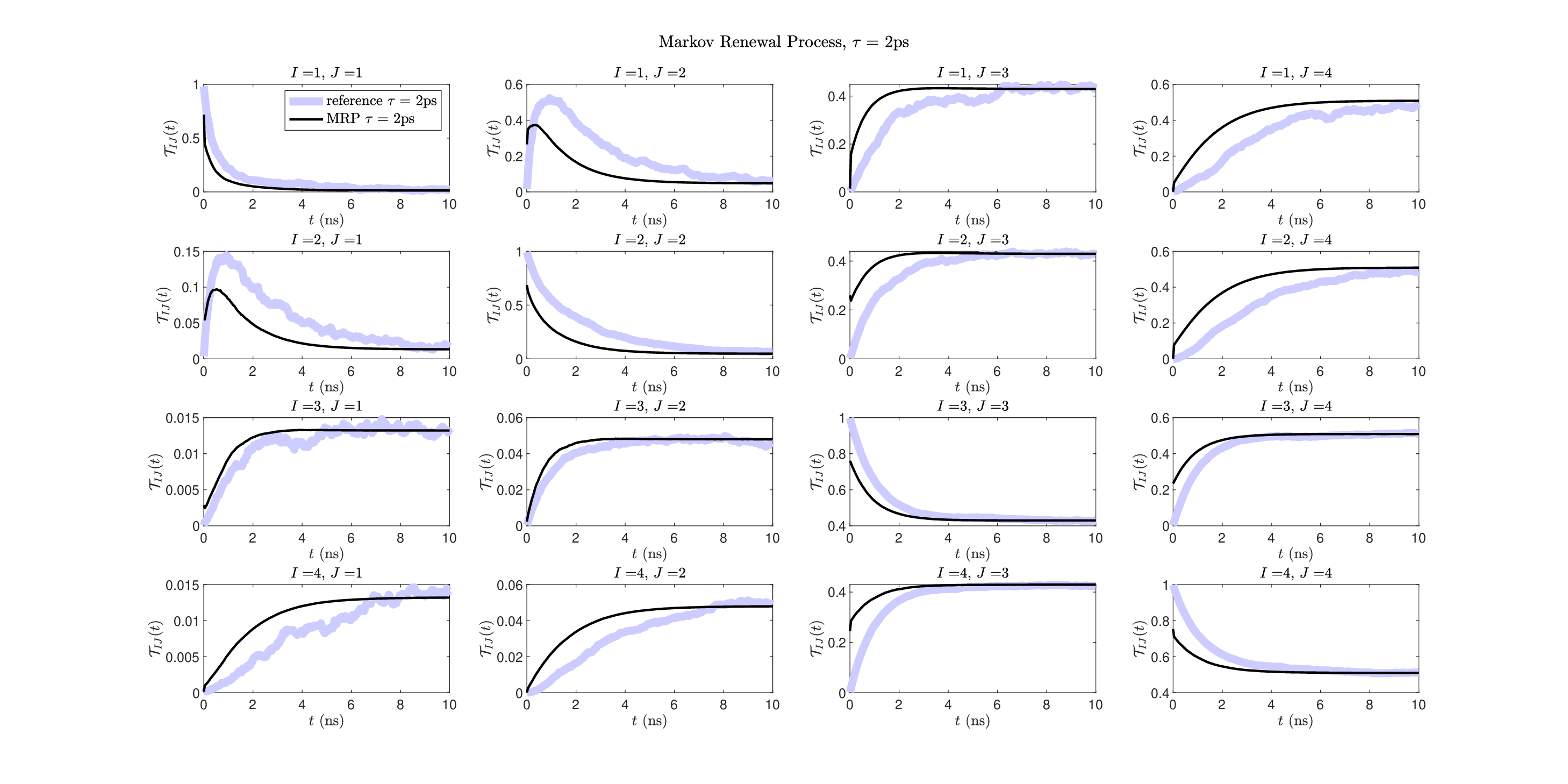}
    }
    
    \subfloat[$\tau = 8$ps]{
    \includegraphics[trim=4.25cm 0cm 4.25cm 0cm, width=0.95\textwidth, clip]{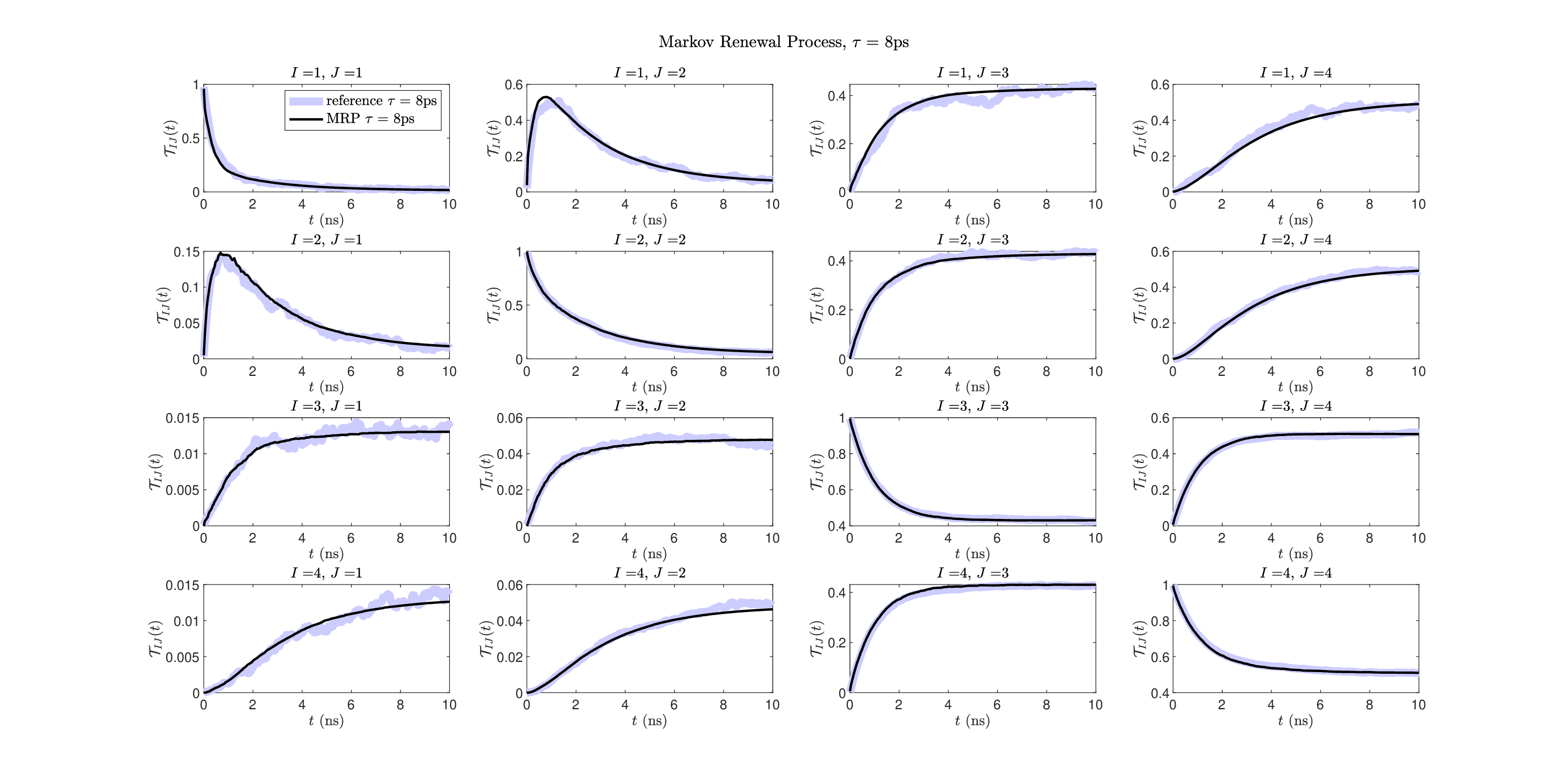}
    }
    \caption{Verifying that $R(t)$ is approximately a MRP for large enough decorrelation times, for PCCA states. Plotted are reference transition probabilities computed from simple counts of $R(t)$, for $\tau = 2$ ps in (a) and $\tau = 8$ ps in (b), compared to probabilities computed from the renewal equation~\eqref{eq:renewal}. (In the renewal equation, the jump probability matrix, $\PP$, is similarly computed from simple counts.) The (constant) decorrelation time must be chosen long enough to allow local equilibration within the macrostates. There is significant disagreement using $\tau = 2$ ps in (a), while the larger value, $\tau = 8$ ps, in (b) gives good agreement without being unnecessarily large.}
    \label{fig:mrp_conv}
\end{figure*}

\begin{figure}
\centering
\subfloat[PCCA States, $\tau = 8$ ps]{\includegraphics[trim=1cm 0cm 1cm 0cm, width=0.45\textwidth, clip]{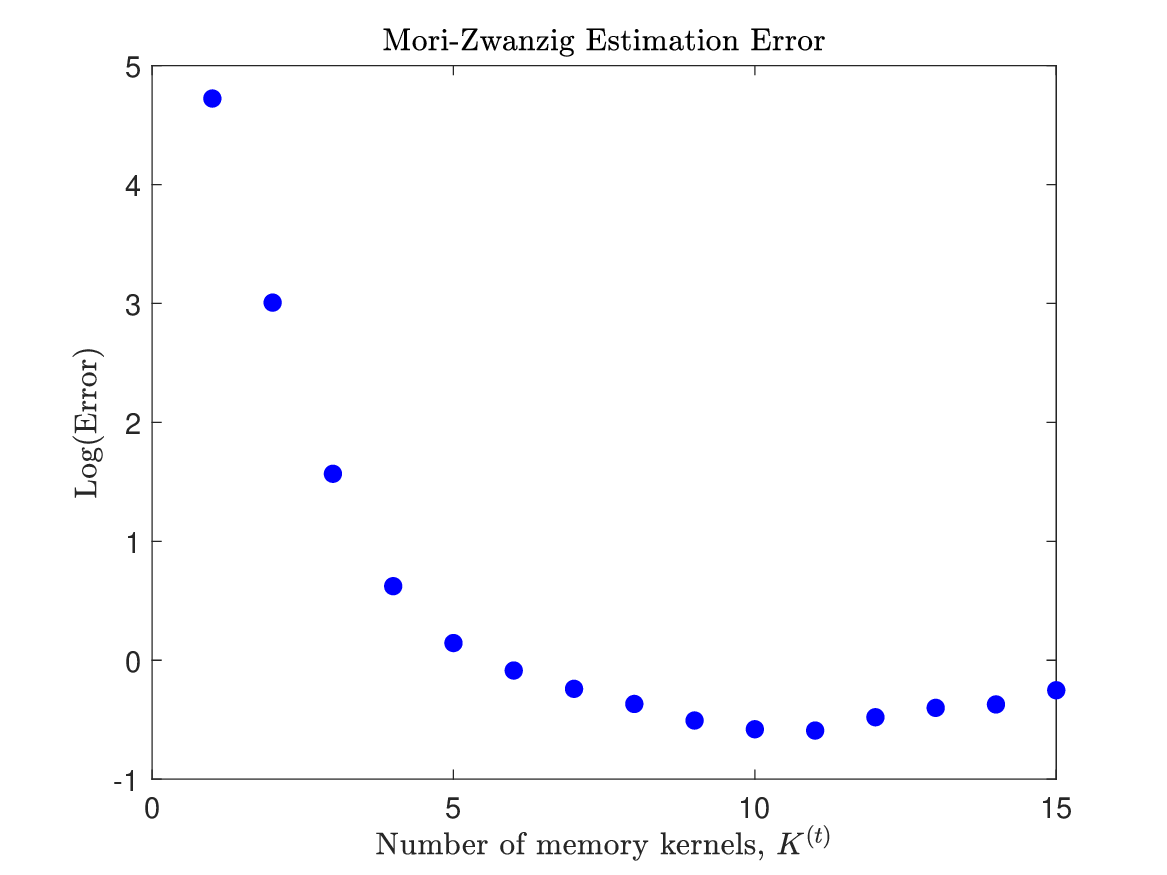}}

\subfloat[Rectangular States, $\tau = 30$ ps]{
\includegraphics[trim=1cm 0cm 1cm 0cm, width=0.45\textwidth, clip]{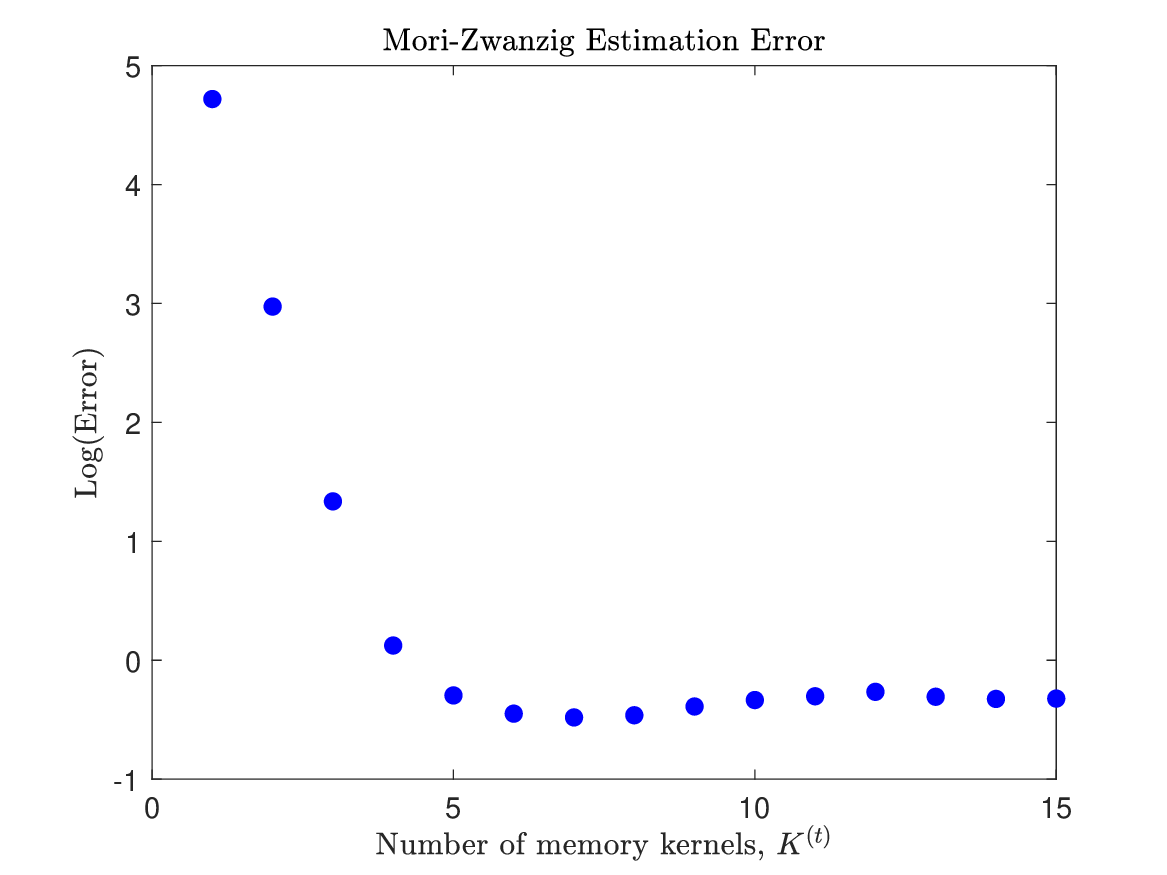}
}
    \caption{The error in our method vs. the number of memory kernels. 
    Memory kernels are estimated at multiples of $\tau$ and the error is defined by~\eqref{eq:Cramer}. 
    The cutoff times, (a): $t_{max}=120$ ps  and (b): $t_{max}=900$ ps, are chosen by applying the rule $t_{mem} = 0.5\times t_{max}$ to the largest $t_{mem}$ pictured (for instance in (b)  $t_{mem} = 450$ ps, corresponding to $15$ memory kernels).}
    \label{fig:mz_conv100}
\end{figure}

\subsection*{Practical considerations}

Our method requires a choice of macrostates and of scalar parameters $\tau$, $t_{mem}$ and $t_{max}$. 
Here, we discuss how these parameters might be chosen. Briefly, the microstates should 
be chosen as metastable 
states associated with 
timescales of interest; 
the parameter $\tau$ should 
be large enough for the Markov property to (nearly) hold, 
but no larger; and $t_{mem}$ 
and $t_{max}$ should be 
as large as needed to 
accurately parametrize 
the model, given constraints 
on how much data is available. 
We discuss all this in more detail below. In this 
discussion, as in the numerical simulations, we assume that all the decorrelation times equal $\tau$, that is, $\tau_I = \tau$ for each macrostate $I$.

We first consider $\tau$, $t_{mem}$ and $t_{max}$. With enough data, increasing $\tau$, $t_{mem}$ and $t_{max}$ will systematically improve results; in practice, though, there are tradeoffs. (Caveat: a too large $\tau$ causes modeling problems; see below.) Clearly, there need to be enough sampled transitions at each time lag. That is, we need enough samples of $\T_{IJ}(t)$ for each $I,J$ and $t \le t_{max}$. So for example, if data comes in the form of many short trajectories of $X(t)$, then increasing $t_{max}$ lowers transition counts, and can improve model fidelity only to the extent that the number of sampled transitions does not get too low. The parameter $t_{mem}$ defines the number of memory kernels, and we found good results 
when pairing it to $t_{max}$ using the rule $t_{mem} \approx 0.5 \times t_{max}$. 
In practice, $t_{max}$ (and/or $t_{mem}$) could be chosen with standard techniques like cross-validation.

The macrostates and the parameter $\tau$ are more fundamental (though they are also subject to similar considerations concerning transition counts). Unlike $t_{mem}$ and $t_{max}$, which are parameters used to obtain the memory kernels which generate an {\em approximation} of $R(t)$, the macrostates and $\tau$ actually {\em define} $R(t)$. 
They must be chosen carefully to yield good results. For a given set of macrostates, a minimum value of $\tau$ is set by the requirement that $R(t)$ is approximately a MRP; the required value can be found empirically by using a plot like Figure~\ref{fig:mrp_conv} (we simply chose one ``by eye'' from such plots). Good 
macrostates are ones in which decorrelation 
occurs on a time scale much smaller than the 
typical escape time -- i.e., good macrostates 
are metastable~\cite{lelievre2015accelerated}. In practice, they could 
be chosen by standard techniques like PCCA~\cite{chodera2014markov}.

A bad choice of macrostates cannot be rescued by a good choice of $\tau$. Indeed, $R(t)$ does not retain any events that occur on timescales smaller than $\tau$. As a result, if $\tau$ is close or larger than typical transition times between macrostates, then $R(t)$ can miss such transitions (as shown in Figure~\ref{fig1}), resulting in a potentially accurate but uninformative model. A good choice of both the macrostates and of $\tau$ is therefore important. 
For the purposes of this article, we think of the macrostates as already being given, and we choose $\tau$ by looking at plots like Figure~\ref{fig:mrp_conv}, 
increasing $\tau$ until we find a good match.

Figure~\ref{fig:msm_comp} shows 
an ordinary Markov model
based on PCCA states. These 
PCCA states are the same as reference~\cite{agarwal2020arbitrarily}. A lag of $1500$ ps is needed for 
accuracy comparable to 
our methods. (Compare with Figure~\ref{fig:mz_comp}(a).) This lag is on the order of the longest mean transition time, roughly $1000$ ps.  Particularly for macrostates 1 and 2, this lag sacrifices knowledge of shorter timescale (but still physically relevant) state-to-state transitions. Although these 
states are considered very good (Markovian) states, our methods still provide significant improvement over Markov models, as illustrated in 
Figure~\ref{fig:mz_comp}.
% This fact is further emphasized when considering the longest median transition times are less than $400$ ps.

Figure~\ref{fig:mz_comp}(a) shows
results from our methods when using 
the PCCA states. There, we use a decorrelation 
time $\tau = 8$ ps. 
This serves as the fundamental 
time step of our coarse-grained model, and is small enough 
that few transitions are missed. 
To build our model, we use many short trajectories of length $112$ ps, smaller than the shortest mean transition time of $175$ ps. (This trajectory length corresponds to using $\tau = 8$ ps, with $t_{mem} = 7$ memory kernels and $t_{max}=2\times t_{mem}$.) 
In contrast, a similarly accurate Markov model in Figure~\ref{fig:msm_comp} requires 
trajectories of length $1500$ ps. Recall that the 
longest mean transition time is around $1000$ ps. 
In sum, the renewal model requires significantly shorter trajectories and is 
more accurate than the Markov model on all timescales. 

% show in Figure~\ref{fig:msm_comp} (compared to Figure~\ref{fig:mz_comp}a) that 
% our methods can improve on Markov models 
% based on PCCA states. 
% We also show in Figure~\ref{fig:mz_comp} that our methods can work on more naively chosen (but not terrible) states as in Figure~\ref{fig:states}, although in that case some transitions will be missed due to relatively large $\tau$. 

Figure~\ref{fig:mz_comp}(b) shows 
analogous results for 
unphysical macrostates 
(defined as equal rectangles in $\phi$-$\psi$ coordinates, divided by the lines $\phi = 0, \pm \pi$ and $\psi = 0, \pm \pi$). Although these states are no longer metastable, results are similar to Figure~\ref{fig:mz_comp}(a). (In this case 
$\tau$ 
needs to be larger, however, resulting 
in our model missing some transitions, as discussed above.) We find good accuracy when $\tau = 30$ ps, $t_{mem} = 15$ memory kernels, and trajectories have length $900$ ps.
A Markov model would require a lag of $5000$ ps for similar 
accuracy. 
Smaller Markov model lags of $\sim 1000$ ps result in wildly inaccurate estimates for even a few time steps' prediction. 

% To achieve accuracy comparable with Figure~\ref{fig:mz_comp},  MSM approximations needed lag times 
% on the same order as the longest mean transition time. 

% Conversely, the Renewal model adeptly captures dynamics by selecting a decorrelation time, which at worst corresponds to the quickest mean transition times. % mats look into {\color{blue} if the shortest and longest mean transition time is closer, is there the same improvement (even with separation of scale)?}

%In more detail, {\bf Mats: Describe simulations... talk about relation between transition times of $X(t)$, $\tau$, and $t_{max}$, $t_{mem}$. Cover both cases (good and bad states) }

% We conclude this section by briefly mentioning the the large data limit (infinitely many, relatively short trajectories). In this limit, aside from the modeling limitations associated with $\tau$, the accuracy of our methods are mainly limited by the linear solve for the kernels. Since the structure of the linear system admits efficient solvers (e.g., RPCholesky~\cite{chen2022randomly,diaz2023robust}), it should scale well to relatively large numbers of states. For example, the number of macrostates multiplied by the number of memory kernels could be $1$ million on a modern laptop.

\section{Discussion}

The methodology introduced in this paper allows for the systematic exploitation of a rich set of trade-offs between compactness, expressiveness, and accuracy. It is particularly well suited to cases where the system contains a relatively small number of metastable states, but where metastability is insufficient for a Markovian assumption to be accurate. In contrast to conventional approaches like Markov State Models, where accuracy can be improved by increasing the number of states at the cost of interpretability, the accuracy of the  approach proposed is instead controlled by increasing the decorrelation time, to ensure the convergence to a MRP. Doing so however comes with its own trade-off, as the expressiveness of $R(t)$ decreases when the decorrelation time exceeds the shortest transition time. However, the geometric convergence rate to an MRP makes this trade-off particularly advantageous as a small increase in decorrelation time yields a large increase in accuracy. Therefore, even for modestly metastable systems, it should be possible to produce very accurate MRPs using decorrelation times that are short compared to typical transition times, hence minimizing the loss of kinetic information. In this situation, the MZ approach described above will also yield a compact representation in terms of a limited number of kernel matrices. As shown above, this approach allows one to obtain compact and accurate models even with sub-optimal state definitions, which is very useful given that optimizing state definitions in high dimension is generally difficult. That being said, the approach cannot fix state definitions where most of the states are not at least somewhat metastable, as accuracy would demand very long decorrelation times, which would then entail low expressiveness.
It is arguable, however, that no representation in terms of jump processes would be appropriate in such a scenario.

\begin{acknowledgments}
D. Aristoff and M. Johnson gratefully acknowledge support from the National Science Foundation via Award No. DMS 2111277.
D. Perez was supported by the Laboratory Directed Research and Development program of Los Alamos National Laboratory under project number 20220063DR. Los Alamos National Laboratory is operated by Triad National Security, LLC, for the National Nuclear Security Administration of U.S. Department of Energy (Contract No. 89233218CNA000001).

D. Aristoff and M. Johnson acknowledge illuminating 
discussions with D.M. Zuckerman, J. Copperman, J. Russo, G. Simpson, and R.J. Webber.
\end{acknowledgments}

\appendix

\section{Convergence to a Markov renewal process}\label{sec:appendix_MR}

We begin by introducing some notation. Let
\begin{equation*}
E_{IJ}(s,t) = \{R(s+t) = J, \,R(s') = I, \,s \le s' < s+t\}
\end{equation*}
be the event of switching to from $I$ to $J$ after a time $t$, starting from time $s$. Let
\begin{equation*}
E_J(t) = \{R(t) = J\},\qquad  E_J^c(t) = \{R(t) \ne J\}
\end{equation*}
be the events that $R(t) = J$ and $R(t) \ne J$, respectively. 

We use $\sim$ to indicate equality in distribution; for example, $X(s)\sim \eta_I$ indicates that 
$X(s)$ is distributed as $\eta_I$.

The following result demonstrates convergence in distribution of $R(t)$ to a Markov renewal process as the decorrelation times grow.
\begin{theorem}[Exactness of renewal equation]\label{thm:RE}
Assume that each macrostate $I$ has a QSD $\eta_I$, and assume that~\eqref{eq:QSD_conv} holds. Define
\begin{equation}\label{eq:defT}
    \T_{IJ}(t) = {\mathbb P}(E_J(s+t)|E_I(s),\,X(s) \sim \eta_I).
\end{equation}
Then $\T(s,t)$ defined by~\eqref{eq:defTst} converges to $\T(t)$ defined by~\eqref{eq:defT} as $\min_I \tau_I \to \infty$. Moreover, 
the limit $\T(t)$ is 
the unique solution to
the renewal equation~\eqref{eq:renewal} when $\PP$ is defined by
\begin{align}\begin{split}\label{eq:TPF}
\PP_{IJ}(t) &=\delta_{I \ne J}{\mathbb P}(E_{IJ}(s,t)|E_I(s),\,X(s) \sim \eta_I).
\end{split}
\end{align} 
\end{theorem}

\begin{proof}
Using the law of total probability, 
\begin{align}\begin{split}\label{eq:tot_prob}
\T_{IJ}(s,t) &= {\mathbb P}(E_J(s+t)|E_I^c(s_-),\, E_I(s)) \\
&= \sum_{K\ne I}\sum_{0<r \le t} {\mathbb P}(E_J(s+t)|E_I^c(s_-),E_{IK}(s,r)) \\
&\qquad \qquad\times {\mathbb P}(E_{IK}(s,r)|E_I^c(s_-),\,E_I(s)) \\
&+ \delta_{I = J} {\mathbb P}(E_{II}(s,t)|E_I^c(s_-),\,E_I(s)).
\end{split}
\end{align}
Let $\sigma = \min_I \tau_I$, and in the notation of~\eqref{eq:QSD_conv}, define
\begin{equation*}
    c = \max_I c_I, \quad \delta = \max_I \delta_I.
\end{equation*}
Using~\eqref{eq:QSD_conv},~\eqref{eq:defT}, and the Markov property of $X(t)$,
\begin{align}\begin{split}\label{eq:Tlong}
\T_{IJ}(s,t)  
&={\mathbb P}(E_J(s+t)\,|\,E_I^c(s_-),\,E_I(s)) \\
&=  \int{\mathbb P}(E_J(s+t)\,|\,E_I^c(s_-),\,E_I(s),\,X(s) = x) \\
&\qquad\quad  \times {\mathbb P}(X(s) \in dx\,|\,E_I^c(s_-),\,E_I(s)) \\
&=  \int {\mathbb P}(E_J(s+t)\,|\,E_I(s),\,X(s) = x)\nu_I(dx) + \epsilon_I \\
&= \T_{IJ}(t) + \epsilon,
\end{split}
\end{align}
where $|\epsilon| \le c \delta^\sigma$. Similar calculations show that 
\begin{align*}
&    \delta_{I \ne K}{\mathbb P}(E_J(s+t)|E_I^c(s_-),E_{IK}(s,r)) = \T_{KJ}(t-r) + \epsilon \\
& \delta_{I \ne K}{\mathbb P}(E_{IK}(s,r)|E_I^c(s_-),\,E_I(s)) = \PP_{IK}(r) + \epsilon \\
& {\mathbb P}(E_{II}(s,t)|E_I^c(s_-),\,E_I(s)) =\F_{II}(t)  + \epsilon,
\end{align*}
where each $\epsilon$ is different but $|\epsilon| \le c \delta^\sigma$, and 
\begin{align*}
    \F_{IJ}(t) &= 
\delta_{I=J} {\mathbb P}(E_{II}(s,t)|E_I(s),\,X(s)\sim \nu_I)  \\
&= \delta_{I=J} \sum_L \sum_{s>t} \PP_{IL}(s).
\end{align*}
Combining the previous three displays with~\eqref{eq:tot_prob},
\begin{equation}\label{eq:renewal_approx}
\T(s,t) = O(t\delta^\sigma) + \sum_{0<r \le t}\PP(r)\T(s+r,t-r) + \F(t).
\end{equation}

Now from~\eqref{eq:Tlong}, we 
conclude that $\T(s,t)$ converges to $\T(t)$ as $\sigma \to \infty$. Meanwhile, using~\eqref{eq:renewal_approx}, it is readily shown from a standard renewal equation representation~\cite[Proposition 4.2]{cinlar1975exceptional} that $\T(t)$ is the unique solution to equation~\eqref{eq:renewal}.
\end{proof}

The proof shows that the convergence 
rate is geometric in $\sigma$ on finite time intervals, suggesting 
that large decorrelation times are 
not needed in order to model $R(t)$ 
as a Markov renewal process, at 
least for reasonably defined states.

\section{Actions of projector and Markov kernels}\label{sec:actions}

Below, we introduce another process $C(t)$ that counts the consecutive 
time that $X(t)$ has spent in its 
current macrostate, where the count stops at $\tau_J$ if $X(t) \in J$. 

To develop the Mori Zwanzig theory, we introduce the augmented Markov chain 
$(X(t),R(t),C(t))$ on 
augmented states $(x,I,s)$, where 
$x$ and $I$ represent the current 
values of $X(t)$ and $R(t)$, and $s$ is the consecutive 
time that $X(t)$ has spent in 
the macrostate in which it currently resides, up to the decorrelation time. 
%(Note that $X(t)$ can reside in a macrostate different from the current value of $R(t)$; see Figure~\ref{fig1}.) 
This Markov chain has time step $\tau$.

Below, let ${\mathbb P}^{x,I,s}$ denote  probability for the 
augmented Markov chain that 
starts at $(X(0),R(0),C(0)) = (x,I,s)$. 
Let $T$ be the Markov kernel of this 
augmented chain,
\begin{align}\begin{split}\label{eq:T}
&T(x,I,s; dy,J,t) \\
&\qquad = {\mathbb P}^{x,I,s}[(X(\tau),R(\tau),C(\tau)) = (dy,J,t)].
\end{split}
\end{align}

We will also make use of more broadly defined kernels $S(x,I,s;dy,J,t)$ by 
relaxing the nonnegativity and unit normalization properties of $T$. Specifically, such a kernel acts on functions $f = f(x,I,s)$ of augmented space according to 
the rule 
\begin{equation*}
Sf(x,I,s) = \int \sum_{J,t} S(x,I,s;dy,J,t)f(y,J,t).
\end{equation*}
% Similarly, such kernels act on measures 
% $\mu = \mu(dx,I,s)$ of augmented space by 
% \begin{equation*}
% \mu S(dx,I,s) = \int_y \sum_{J,t} \mu(dy,J,t)S(y,J,T;dx,I,s).
% \end{equation*}

We define a projector $P$ on functions $f = f(x,I,s)$ of augmented states, that is, a 
mapping satisfying $P^2 = P$, by
\begin{equation}\label{eq:projector}
Pf(x,I,s) = \int \eta_I(dz)f(z,I,\tau_I).
\end{equation}

\section{Principal and orthogonal dynamics, and Markovian case}\label{sec:dynamics}

The Mori-Zwanzig theory is 
characterized by a {\em principal} 
and {\em orthogonal} dynamics. 
The principal dynamics is driven 
by $PT$, defined by
\begin{equation*}
PT(x,I,s;dy,J,t) = \int \eta_I(dz)T(z,I,\tau_I;dy,J,t).
\end{equation*}
The orthogonal dynamics is driven by $QT$, 
where $Q = \textup{Id} - P$ and $\textup{Id}$ is the identity operator; 
that is, $QT = T - PT$. 

We consider a special {\em Markovian case}, 
in which the underlying dynamics instantaneously 
reaches the QSD in whatever macrostate it 
resides in, with associated decorrelation times $\tau_I = 0$ for all $I$. In this case, $T = PT$, so the orthogonal dynamics vanish, $QT= 0$, and 
all but one of the 
memory kernels is zero; 
see Appendix~\ref{sec:appendix_MZ}.

\section{Derivation of the Mori Zwanzig equation}\label{sec:appendix_MZ}

The following lemma applies to any transition 
kernel $T$ and projector $P$, although 
we have in mind the Markov kernel $T$ in~\eqref{eq:T} and the projector $P$ 
in~\eqref{eq:projector}. 

\begin{lemma}\label{lem:gen_MZ}
For any projector $P$ and its complementary projector $Q = \textup{Id} - P$, where $Id$ is the identity mapping, we have 
\begin{equation}\label{eq:general_MZ}
P T^{n} = \sum_{m=1}^{n}K(m)P T^{n-m}  +  F(n),
\end{equation}
where $K(n) = PT(QT)^{n-1}$ and $F(n) = PT(QT)^{n-1}Q$.
\end{lemma}
\begin{proof}
Start with the self-evident equations
\begin{align}
    P T^{n+1} &= P T P T^n  + P T Q T^n \label{MZeq1} \\
Q T^{n+1} &= Q T P T^n  + Q T Q T^n. \label{MZeq2}
\end{align}
Using induction in~\eqref{MZeq2},
\begin{equation*}
    Q T^n  = \sum_{m=1}^{n} (Q T)^{m}P T^{n-m}  + (Q T)^n Q .
\end{equation*}
Plugging this back into~\eqref{MZeq1} yields the result.\end{proof}

Below, we will make use of functions $\chi_J$ defined by 
\begin{equation*}
\chi_J(x,I,s) = \delta_{I=J}.
\end{equation*}

\begin{theorem}[Exactness of MZ equation]\label{thm:MZeqn}
Let $K(n)$ be as in Lemma~\ref{lem:gen_MZ}, where $P$ is the projector from~\eqref{eq:projector} and $T$ is defined in~\eqref{eq:T}. Define
\begin{equation}\label{eq:matrix_kernel}
    \K_{IJ}(n\tau) := K(n) \chi_J(x,I,s).
\end{equation}
Then, with $\T(t)$ as in~\eqref{eq:defT}, 
\begin{equation}\label{eq:matrixMZ}
    \T(n\tau) = \sum_{m=1}^n \K(m\tau)\T((n-m)\tau).
\end{equation}
\end{theorem}
\begin{proof}
Multiply~\eqref{eq:general_MZ} on the right by $\chi_J(x,I,s)$. Note that $P\chi_J = \chi_J$, so that $Q\chi_J =0$ and $F(n)\chi_J(x,I,s) = 0$. Thus,
\begin{equation}\label{eq:multiplied}
  P T^{n}\chi_J(x,I,s) = \sum_{m=1}^{n}K(m)P T^{n-m}\chi_J(x,I,s).
\end{equation}

Recalling $\T(t)$ defined in~\eqref{eq:defT}, we compute
\begin{align*}
PT^n\chi_J(x,I,s) &= \int \eta_I(dx)T^n\chi_J(x,I,\tau_I) \\
&= \int \eta_I(dx)\E^{x,I,\tau_I}[\chi_J(X(n\tau),R(n\tau),C(n\tau))] \\
&= \int \eta_I(dx) {\mathbb P}^{x,I,\tau_I}[R(n\tau) = J] \\
&= \T_{IJ}(n\tau).
\end{align*}
Below, write 
$S_m = T(QT)^{m-1}$, 
and note that $PT^{n-m}\chi_J(y,L,t)$ does not 
depend on $y$ or $t$. Thus,
\begin{align*}
  &\sum_L  \K_{IL}(m\tau) \T_{LJ}((n-m)\tau) \\
  &= 
  \sum_L K(m)\chi_L(x,I,s) P T^{n-m}\chi_J(y,L,t) \\
  &= \sum_L\int \eta_I(dz)\left[\int \sum_t S_m(z,I,\tau_I;dy,L,t)\right]PT^{n-m}\chi_J(y,L,t) \\
  &=\int \eta_I(dz)\left[\int \sum_{L,t} S_m(z,I,\tau_I;dy,L,t)PT^{n-m}\chi_J(y,L,t)\right] \\
  &= K(m)PT^{n-m}\chi_J(x,I,s).
\end{align*}
Combining the last two displays with~\eqref{eq:multiplied} gives~\eqref{eq:matrixMZ}. 
\end{proof}

Next, we show that all 
but one of the memory kernels 
vanishes in the case where $R(t)$ 
is Markovian.
\begin{theorem}\label{thm:Markov}
   Suppose that $\tau_I = 0$ for all $I$ and that
\begin{equation*}
    T(x,I,s;dy,J,t) = \int \nu_I(dz)T(z,I,\tau_I;dy,J,t).
\end{equation*}
Then 
$\K(n\tau) = 0$ for $n > 1$.
\end{theorem}
\begin{proof}
    The assumption on $T$ implies that $PT = T$, so $QT = 0$ and 
    the result follows from the formula $$\K_{IJ}(n\tau) = PT(QT)^{n-1}\chi_J(x,I,s).$$
\end{proof}

\section{Minimizing the loss function}\label{sec:appendix_linear}

The gradient of the loss function~\eqref{eq:loss} is 
\begin{align*}
        \nabla_{\K(t)} {\mathcal L}(\K) 
        &= \sum_{r \le t_{max}} \T(r)\T(r-t)^T \\
        &\qquad - \sum_{0 < s \le t_{mem}} \K(s)\sum_{r\le t_{max}} \T(r-s)\T(r-t)^T,
\end{align*}
where by definition $\T(s) = 0$ for $s<0$.

This immediately leads to the linear system reported in~\eqref{eq:linear}.

\providecommand{\noopsort}[1]{}\providecommand{\singleletter}[1]{#1}%

\end{document}